\documentclass[11pt,a4paper]{amsart}
\pdfoutput=1
\usepackage{mathptmx,graphicx,overpic}
\DeclareMathOperator{\led}{led}
\DeclareMathOperator{\inc}{inc}
\DeclareMathOperator{\dist}{dist}
\DeclareMathOperator{\width}{width}
\newcommand{\mbf}{\mathbf}
\newcommand{\mc}{\mathcal}
\newtheorem{thm}{Theorem}
\newtheorem{cor}[thm]{Corollary}
\newtheorem{lem}[thm]{Lemma}
\newtheorem{prop}[thm]{Proposition}

\newtheorem*{thm*}{Theorem}
\newtheorem*{cor*}{Corollary}
\newtheorem*{lem*}{Lemma}
\newtheorem*{prop*}{Proposition}

\theoremstyle{definition}

\newtheorem*{defn*}{Definition}

\newtheorem*{prob*}{Problem}

\theoremstyle{definition}

\theoremstyle{definition}

\newtheorem*{conj*}{Conjecture}

\newtheorem*{example*}{Example}

\begin{document}
\title{The Reversal Ratio of a Poset}
\author{Graham Brightwell}
\email{G.R.Brightwell@lse.ac.uk}
\address{Department of Mathematics\\London School of Economics\\Houghton Street\\London\\WC2A 2AE\\United Kingdom}
\author{Mitchel T.\ Keller}
\thanks{This research was conducted while the second author was a
  Visiting Fellow at the London School of Economics supported by a Marshall
  Sherfield Fellowship.}
\email{kellermt@wlu.edu}
\address{\emph{Current address}: Department of Mathematics\\Washington
  and Lee University\\ 204 W Washington Street\\Lexington, VA
  24450\\United States}
\date{9 December 2013}

\begin{abstract}
  Felsner and Reuter introduced the linear extension diameter of a
  partially ordered set $\mbf{P}$, denoted $\led(\mbf{P})$, as the
  maximum distance between two linear extensions of $\mbf{P}$, where
  distance is defined to be the number of incomparable pairs appearing
  in opposite orders (reversed) in the linear extensions. In this paper, we
  introduce the reversal ratio $RR(\mbf{P})$ of $\mbf{P}$ as the ratio
  of the linear extension diameter to the number of (unordered)
  incomparable pairs. We use probabilistic techniques to provide a
  family of posets $\mbf{P}_k$ on at most $k\log k$ elements for which the
  reversal ratio $RR(\mbf{P}_k)\leq C/\log k$, where $C$ is an absolute constant.
  We also examine the questions of bounding the reversal ratio in terms of
  order dimension and width.
\end{abstract}

\maketitle

\section{Introduction}\label{sec:intro}

Let $\mbf{P}=(X,P)$ be a finite poset that is not a chain, and let
$L_1$ and $L_2$ be two linear extensions of $\mbf{P}$. The
\emph{distance between $L_1$ and $L_2$}, denoted $\dist(L_1,L_2)$, is
defined to be the number of pairs $(x,y)\in X\times X$ such that $x<_{L_1} y$
and $y<_{L_2} x$. Such an incomparable pair is said to be
\emph{reversed} by the linear extensions $L_1$ and $L_2$. In
\cite{felsner:led}, Felsner and Reuter introduced the \emph{linear 
  extension diameter} of $\mbf{P}$ as
\[\led(\mbf{P}) = \max_{L_1,L_2}\ \dist(L_1,L_2),\]
where the maximum is taken over all pairs of linear extensions of
$\mbf{P}$. The linear extension diameter can also be viewed as the
diameter of the \emph{linear extension graph} of $\mbf{P}$: this graph has
a vertex corresponding to each linear extension of $\mbf{P}$, with two
vertices adjacent if the corresponding linear extensions differ only
in the transposition of a single adjacent pair of elements.

If we let $\inc(\mbf{P})=|\{(x,y)\in X\times X \colon x\|_P y\}|/2$,
the number of \emph{unordered} incomparable pairs of $\mbf{P}$, then
it immediately follows that $\led(\mbf{P})\leq \inc(\mbf{P})$, with
equality if and only if $\mbf{P}$ has dimension $2$. We define the
\emph{reversal ratio} of a pair $(L_1,L_2)$ of linear extensions of
$\mbf{P}$ to be
$RR(\mbf{P};L_1,L_2)=\dist(L_1,L_2)/\inc(\mbf{P})$. The \emph{reversal
  ratio of the poset $\mbf{P}$} is then $RR(\mbf{P}) =
\led(\mbf{P})/\inc(\mbf{P}) = \max_{L_1,L_2} RR(\mbf{P};L_1,L_2)$.

A question of interest about linear extension diameter is to determine
if there exists a constant $c>0$ such that $\led(\mbf{P})\geq
c\inc(\mbf{P})$, or equivalently if $RR(\mbf{P})\geq c$ for all
$\mbf{P}$. Let us briefly motivate this by examining cases where there
is such a constant.  First, if $\mbf{P}$ has low dimension, then
taking some pair of linear extensions from a small realizer reverses a
large proportion of the incomparable pairs.  Second, if the width of
$\mbf{P}$ is a positive proportion of the number of elements, then
some pair of linear extensions inverting the large antichain
automatically has a large reversal ratio.  Thus, a poset $\mbf P$ with
small reversal ratio must have many incomparable pairs, low width, and
unbounded dimension.  One possible candidate with these properties is
the $d$-dimensional Boolean lattice $\mbf{Q}_d$, but Felsner and
Reuter~\cite{felsner:led} found a pair of linear
extensions of $\mbf{Q}_d$ reversing approximately half of the
incomparable pairs of $\mbf{Q}_d$. Felsner and
Massow~\cite{felsner:led-downlat} later showed that this pair is
optimal, establishing that $\led(\mbf{Q}_d) = 2^{2d-2} -
(d+1)2^{d-1}$. More generally, in \cite{felsner:led-downlat}, Felsner
and Massow give a polynomial-time algorithm for computing the linear
extension diameter of downset lattices of $2$-dimensional posets; an
analysis of their proof shows that $RR(\mbf{P}) \ge 1/2$ for all such
posets.  By contrast, Brightwell and Massow~\cite{brightwell:led-diampairs} showed that it is NP-complete
to compute the linear extension diameter of an arbitrary poset.

However, both \cite{felsner:led-downlat}
and \cite{felsner:led} reference personal communication from
Brightwell that there is no such constant.  Brightwell's idea was to take a
random graph order $P_{n,p(n)}$ (see~\cite{albert-frieze, bollobas-brightwell:width-rgos,
bollobas-brightwell:dimension-rgos}) for an appropriately chosen
$p(n)$, which is known to have the properties indicated above, and show that
there is a function $r(n) = o(1)$ such that $RR(P_{n,p(n)}) \le r(n)$
w.h.p.\ (i.e., with probability tending to~1 as the size of the poset tends to
infinity).

It turns out to be awkward to analyze exactly how small the reversal
ratio is for random graph orders. In this paper, we examine this
problem further by considering a different model of random posets
$\mbf{P}_k$ and showing that $RR(\mbf{P}_k) \leq 27/\log k$.  (The poset
$\mbf{P}_k$ has width $k$ and about $0.434 k\log k$ elements.)  We also
consider the problems of bounding the linear extension diameter of a
poset of fixed dimension and fixed width.

Readers unfamiliar with terminology and notation should consult
Trotter's monograph \cite{trotter:dimbook}.

\section{Posets with arbitrarily small reversal ratio}\label{sec:prob}

In this section, we show there is no constant $c$ such that
$\led(\mbf{P})\geq c\,\inc(\mbf{P})$ for all posets $\mbf{P}$. To
accomplish this, we define a family of posets whose linear extension
diameter tends to zero.

Let $k$ be a multiple of~3. We say that a bipartite graph with vertex
classes $A$ and $B$, each of size $k$, has the {\em doubling property} if (i)~for every subset $X$
of $A$ of size at most $k/3$, the set $\Gamma(X)$ of neighbors of $X$
has size greater than $2|X|$, and (ii)~for every subset $Y$ of $B$ of
size at most $k/3$, $|\Gamma(Y)| \ge 2 |Y|$.

The doubling property is an instance of a {\em vertex expansion
  property}.  A graph $G$ is said to be an {\em $(\alpha, \rho)$
  vertex-expander} if, for every subset $X$ of $V(G)$ of size at most
$\alpha |V(G)|$, we have $|\Gamma(X)| \ge \rho|X|$.  (Here the only
vertices of $X$ included in $\Gamma(X)$ are those that have a neighbor
in $X$.)  If $G$ is a $(1/3,2)$ vertex-expander on $k$ vertices, take
two copies $A$, $B$ of $V(G)$, and form a bipartite graph $H(G)$ with
vertex classes $A$ and $B$ by joining $a$ and $b$ if and only if the
corresponding vertices in $G$ are adjacent.  Thus the maximum degree
of $H(G)$ is the same as in $G$, and $H(G)$ has the doubling property.

There are explicit constructions known of $(1/3,2)$ vertex expanders:
for instance, {\em Ramanujan graphs} (\cite{lubotzky-phillips-sarnak})
with degree at least~15 are $(1/3,2)$ vertex expanders.  To keep our
proof self-contained, we shall give a short argument showing that a
certain class of random bipartite graphs with maximum degree at
most~10 has the doubling property w.h.p.

Fix an integer $k$ which is a multiple of~3, and a real number
$\varepsilon > 0$.  Given a bipartite graph $G(k)$ on two vertex classes
$A$ and $B$, each of size $k$, with maximum degree at most~$10$,
having the doubling property, we now describe how to construct our posets.
Let $A^{k,\varepsilon}$ be the union of $m=\lfloor \varepsilon \log k \rfloor$ disjoint
sets $A_1,\dots,A_m$, each of size $k$.  We form a poset
$\mbf{P}_{k,\varepsilon} = (A^{k,\varepsilon},P^{k,\varepsilon})$ with
vertex set $A^{k,\varepsilon}$ as follows.  Let $G_i$ be a bipartite
graph isomorphic to $G(k)$ with vertex classes $A_i$ and $A_{i+1}$.
The comparabilities between $A_i$ and $A_{i+1}$ are defined so that
$x_i\in A_i$ is less than $x_{i+1}\in A_{i+1}$ if and only if
$x_ix_{i+1}$ is an edge of $G_i$. The partial order $P^{k,\varepsilon}$ is then
formed by taking the transitive closure of the union of these sets of
comparabilities.  To be specific, $x_0 \in A_i$ is below $x_j \in
A_{i+j}$ if and only if there are vertices $x_\ell \in A_{i + \ell}$
($\ell = 1, \dots, j-1$), such that each $x_\ell x_{\ell+1}$ ($\ell
=0,\dots, j-1$) is an edge of the copy $G_i$ of the graph $G(k)$.

A very similar construction was used by Fox~\cite{fox} in a different context:
he required an example of an $n$-element poset such that no element is comparable
with more than $n^\epsilon$ others, while for each two sets $A$ and $B$ of order 
$C n /\log n$, some element of $A$ is comparable with some element of $B$.  

Before showing that a poset $\mbf{P}_{k,\varepsilon}$ as above has small reversal
ratio, we show the existence of the bipartite graphs $G(k)$.  We
describe the {\em configuration model} for random bipartite graphs.
Given positive integers $r$ and $k$ (in our application, $r$ will be
equal to 10, and $k$ will be sufficiently large and a a multiple of~3),
we take two disjoint vertex sets $A$ and $B$, each of size $k$.  We now
form two $kr$-element sets $\tilde A$ and $\tilde B$: for each vertex
$a$ of $A$, we place $r$ elements $a_1, \dots, a_r$ into $\tilde A$, and
likewise for $\tilde B$.  Next, we take a uniform random matching of
the elements of $\tilde A$ and $\tilde B$: so each element $a_i$ of
$\tilde A$ is paired with exactly one element of $\tilde B$.  Finally,
we form our bipartite graph on $A \cup B$ by putting an edge between
$a \in A$ and $b \in B$ whenever some $a_i$ is paired with some $b_j$.
We denote a random graph generated in this manner $G_r(A,B)$.  The
random graph $G_r(A,B)$ is bipartite, with maximum degree at most $r$.
There may be instances of pairs $(a,b)$ where two different $a_i$ are
paired with some $b_{j(i)}$, resulting in these vertices having degree
less than $r$.  However, if $r$ is fixed and $k$ is large, then there
are w.h.p.\ few such instances.

The configuration model is the basis of the standard approach for
analyzing random $r$-regular graphs.  To get an $r$-regular graph, one
has to avoid the problem of multiple edges, for instance by rejection
sampling, but for our purposes there is no difficulty with using the
random graph $G_r(A,B)$ directly.

We will work with expressions containing $r$ to begin with for greater
clarity. We then substitute $r=10$ to obtain useful numerical
estimates.

\begin{lem}\label{lem:expand}
  Let $G_r(A,B)$ be a random bipartite graph on vertex sets $A$ and
  $B$ of size $k$, chosen according to the configuration model. For
  $k$ sufficiently large and $r=10$, the probability that $G_r(A,B)$
  does not have the doubling property is at most $2/k^5$.
\end{lem}

\begin{proof}
  Take any natural number $x \le k/3$, any subset $X$ of $A$ of size
  $x$, and any subset $Y$ of $B$ of size $2x$.  Now the number of
  pairings in the configuration model such that all neighbors of $X$
  lie in $Y$ is equal to
\[
\binom{2xr}{xr} (xr)! (kr-xr)! = \frac{(2xr)!}{(xr)!} ((k-x)r)!,
\]
and so the probability that all neighbors of $X$ lie in $Y$ is equal to
\[
\frac{\frac{(2xr)!}{(xr)!} ((k-x)r)!}{(kr)!} \le 2 \frac{(2xr/e)^{2xr}}{(xr/e)^{xr}}\frac{((k-x)r/e)^{(k-x)r}}{(kr/e)^{kr}}
= 2 \left[\frac{(2x)^{2x}(k-x)^{k-x}}{x^xk^k}\right]^r.
\]
To prove the lemma, we find an upper bound on the expected number $e_x$ of pairs $(X,Y)$ as above,
for each $x \le k/3$.  This expectation $e_x$ is at most
\begin{align*}
2 \binom{k}{x}\binom{k}{2x}\left[\frac{(2x)^{2x}(k-x)^{k-x}}{x^xk^k}\right]^r
&\leq \left(\frac{ke}{x}\right)^x \left(\frac{ke}{2x}\right)^{2x}
\left[\frac{4^xx^x}{k^x}\left(1-\frac{x}{k}\right)^{k-x}\right]^r\\
&=\frac{x^{(r-3)x}}{k^{(r-3)x}} e^{3x} 4^{(r-1)x} \left(1-\frac{x}{k}\right)^{(k-x)r}.
\end{align*}
Since the last factor above is at most $1$ and we are taking $r=10$, $e_x$ is at most
\[
p(x) := \left[\left(\frac{x}{k}\right)^7 e^3 4^9\right]^x.
\]
For $1 \le x \le k/11$, $p(x)$ is a convex function which is at most
$$
\max ( p(1), p(k/11) )
= \max \left( \frac{e^3 4^9}{k^7}, \left[\frac{e^3 4^9}{11^7}\right]^{k/11} \right)
\le \frac{1}{k^6}
$$
for sufficiently large $k$.  Thus $e_x \le 1/k^6$ for $1 \le x \le
k/11$, provided $k$ is sufficiently large.

For $x/k\geq 1/11$, we retain the factor of
$(1-x/k)^{(k-x)r}$. Writing $\alpha$ for $x/k$, we then have
\[
e_x \le \left[\alpha^7 e^3 4^9(1-\alpha)^{10(1/\alpha-1)}\right]^x.
\]
For $1/11\leq \alpha \leq 1/3$, the inner quantity is less than $3/4$, as shown by the plot
in Figure~\ref{fig:plot}.

\begin{figure}[h]
  \centering
  \includegraphics[scale=0.65]{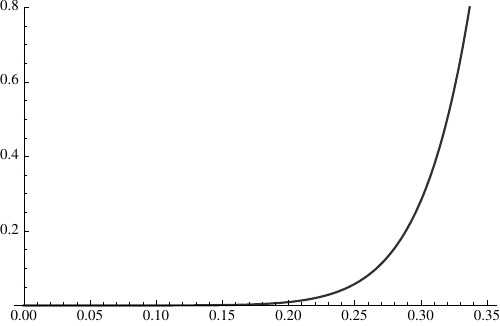}
  \caption{A plot of $\alpha^7 e^3 4^9(1-\alpha)^{10(1/\alpha-1)}$}
  \label{fig:plot}
\end{figure}

Therefore, for $k/11 \le x \le k/3$, we have $e_x \le (3/4)^{k/11} \le 1/k^6$,
for $k$ sufficiently large.  Summing over all values of $x$ up to $k/3$, the probability that
every set $X \subset A$ of size at most $k/3$ has a neighborhood of size greater than $2|X|$ is at least
$1 - 2/k^5$.  We have the same bound for the probability that every set $Y \subset B$ of size at most $k/3$
has a neighborhood of size greater than $2|Y|$, and thus we obtain the
claimed result.
\end{proof}

All we need from Lemma~\ref{lem:expand} is the existence, for all $k$
sufficiently large, of a graph of maximum degree at most~10 with the
doubling property, so that we may construct our example $\mbf{P}_{k,\varepsilon}$ as
stated.

Before bounding $\led(\mbf{P}_{k,\varepsilon})$, we look at $\inc(\mbf{P}_{k,\varepsilon})$.

\begin{prop}\label{prop:inc-Pk}
  For $k$ sufficiently large and $\varepsilon \leq 1/\log r$, the
  number of incomparable pairs in $\mbf{P}_{k,\varepsilon}$ is at least
\[
\frac{r-2}{2(r-1)} \varepsilon^2 k^2 \log^2 k.
\]
\end{prop}

\begin{proof}
  The number of elements of $\mbf{P}_{k,\varepsilon}$ is
  $n= km=k \lfloor\varepsilon \log k\rfloor$. Using just the fact that
  each graph $G_i$ has maximum degree at most $r$, we see that each
  element is comparable to at most
\[
1+ r + r^2 +\cdots + r^{m -1} = \frac{r^m-1}{r-1} \le \frac{k}{r-1}
\]
elements, including itself. Therefore, the number of comparable pairs
is at most $nk/(2(r-1)) \leq \binom{n}{2}/(r-1)$ for $k$ sufficiently
large. Now we see that for $k$ sufficiently large, the number of
incomparable pairs is at least
\[
\frac{r-2}{r-1} \binom{n}{2} \ge \frac{r-2}{2(r-1)}\varepsilon^2 k^2 \log^2 k
\]
as claimed.
\end{proof}

We now turn to estimating the linear extension diameter of $\mbf{P}_{k,\varepsilon}$.

\begin{prop}\label{prop:led-Pk}
Let $k$ be a multiple of~3, and suppose $0 < \varepsilon \le 1$.
The linear extension diameter of $\mbf{P}_{k,\varepsilon}$ is at most
$(31/6) \varepsilon k^2\log k$, for sufficiently large $k$.
\end{prop}

\begin{proof}
  For any linear extension $L$ of $\mbf{P}_{k,\varepsilon}$, we define a second
  linear extension $L'$. For $x\in A_i$ and $y\in A_j$, we let
  $x<_{L'} y$ if $i<j$. If $i=j$, then $x<_{L'} y$ if and only if
  $x<_L y$.  Our plan is to prove an upper bound on $\max_L \dist
  (L,L')$.  Given such a bound, we then note that, for any pair
  $(L_1,L_2)$ of linear extensions, we have
  \[
  \dist(L_1,L_2)\leq
  \dist(L_1,L'_1)+\dist(L'_1,L'_2)+\dist(L'_2,L_2).
  \]
  If $x<_{L'_1} y$ and $y<_{L'_2} x$, then $x,y\in A_i$ for some $i$, so
  $\dist(L'_1,L'_2) \leq m \binom{k}{2} \leq \frac 12 \varepsilon k^2\log k$.
  It thus suffices to prove that $\dist(L,L') \le (7/3) \varepsilon k^2 \log k$, for any
  linear extension $L$ of $\mbf{P}_{k,\varepsilon}$.

  For the linear extension $L$, we need to consider, for each pair $(i,j)$
  with $i<j$, how many instances there are of an element $x$ of $A_i$ appearing
  above an element $y$ of $A_j$ in $L$ -- we say simply that $(x,y)$ is
  {\em reversed in $(A_i,A_j)$}.

  We start with the case where $j=i+2s+1$, i.e., there are an even
  number of levels between $A_i$ and $A_j$.  Consider
  first $s=0$, and let $T\subset A_i$ be the set consisting of the top
  $k/3$ elements of $A_i$ in $L$, and let $T'\subset A_{i+1}$ be the set
  consisting of the bottom $k/3$ elements of $A_{i+1}$ in $L$.
  Suppose that some element of $A_i \setminus T$ is above some element
  $y$ of $A_{i+1} \setminus T'$ in $L$.  Then all the elements of $T$ are above
  all the elements of $T' \cup \{ y\}$, which means there are no edges between $T$
  and $T' \cup\{ y\}$ in $G_i$, contradicting the doubling property.
  Thus every element of $A_i\setminus T$ is below every element of
  $A_{i+1}\setminus T'$ in $L$, and therefore the number of reversed
  pairs $(x,y)$ in $(A_i,A_{i+1})$ is at most
  \[
  |T|\,|A_{i+1}| + |T'|\, |A_i| \le 2k^2/3.
  \]
  For $0<s\leq \lfloor \varepsilon \log k\rfloor /2$, let $T \subset
  A_i$ be the set consisting of the top $\lceil k/(2^s \cdot 3)
  \rceil$ elements of $A_i$ in $L$, and let $T' \subset A_{i+2s+1}$ be
  the set consisting of the bottom $\lceil k/(2^s \cdot 3) \rceil$
  elements of $A_{i+2s+1}$ in $L$.  Again suppose that some element of
  $A_i \setminus T$ is above some element $y$ of $A_{i+2s+1} \setminus
  T'$ in $L$.  Then all the elements of $T$ are above all the elements
  of $T' \cup \{ y\}$ in $L$.  By the doubling property, there are at
  least $2^s |T| \ge k/3$ elements of $A_{i+s}$ above some element of
  $T$ in $L$, and at least $2^s (|T'|+1) > k/3$ elements of
  $A_{i+s+1}$ below some element of $T' \cup \{ y\}$ in $L$, and again
  by the doubling property applied to $(k/3)$-element subsets of $T$
  and $T'\cup\{y\}$ this is not possible.
  Thus, we see that the number of pairs reversed in
  $(A_i,A_{i+2s+1})$ is at most $2k \lceil k/(2^s\cdot 3) \rceil \le k^2/(2^{s-1} \cdot 3) +2k$.
  For fixed $i$, summing over all possible values of $s$, we obtain an upper
  bound of
  \[\sum_{s=0}^{\lfloor \varepsilon\log k\rfloor /2}
  \left( \frac{k^2}{2^{s-1}\cdot 3} + 2k \right) \leq
  \frac{4k^2}{3} - \frac{2k^2}{3} 2^{-\varepsilon \log k/2} + k \varepsilon \log k \leq
  \frac{4k^2}{3}\]
  on the number of reversed pairs involving an element of $A_i$ and an element
  an odd number of layers above $A_i$.  Accounting for all
  values of $i$ gives that we reverse at most $4\varepsilon k^2/3 \log k$ pairs
  in pairs with an even number of layers between them (including zero).

  One approach for counting pairs reversed in $(A_i,A_{i+2s+2})$ is to
  proceed as above, taking $T\subset A_i$ of size $\lceil k/(2^s\cdot
  3)\rceil$ and $T'\subset A_{i+2s+2}$ of size $\lceil
  k/(2^{s+1}\cdot 3)\rceil$.  As above, every reversed pair involves an element
  of either $T$ or $T'$, and so we reverse at most
  \[k\left\lceil\frac{k}{2^s\cdot 3}\right\rceil + k\left\lceil
    \frac{k}{2^{s+1}\cdot 3}\right\rceil\leq \frac{k^2}{2^{s+1}} + 2k\]
  pairs in $(A_i,A_{i+2s+2})$.  For fixed $i$, summing over all $s$
  with $0\leq s\leq \lfloor \varepsilon\log k\rfloor/2$, we have that
  at most $k^2$ pairs involving an element of $A_i$ and an element an even number
  of layers above $A_i$.  Multiplying this by $\varepsilon \log k$ gives that we reverse
  at most $\varepsilon k^2 \log k$ pairs with an odd number of layers between them.

  Therefore,
  \[
  \dist(L,L')\leq \frac{4\varepsilon k^2\log k}{3} + \varepsilon
  k^2\log k = \frac{7}{3}\varepsilon k^2\log k
  \]
  for any linear extension $L$ of $\mbf{P}_{k,\varepsilon}$, as required.  We are
  now able to conclude that $\led(\mbf{P}_{k,\varepsilon})\leq (31/6)\varepsilon
  k^2\log k$.
\end{proof}

As a consequence of Propositions~\ref{prop:inc-Pk} (with $r=10$) and
\ref{prop:led-Pk}, we have the following corollary.

\begin{cor}
Whenever $k$ is a sufficiently large multiple of~3, and $0 <\varepsilon \le 1/\log 10$, we have
\[
RR(\mbf{P}_{k,\varepsilon}) \le \frac{93}{8\varepsilon}\frac{1}{\log k}.
\]
\end{cor}

In particular, setting $\varepsilon = 1/\log 10$ gives us examples of posets $\mbf{P}_k$, for arbitrarily large
$k$, with $n = k \lfloor \log k / \log 10 \rfloor$ elements, and
\[
RR(\mbf{P}_k) \le \frac{93 \log 10}{8} \frac{1}{\log k} \le \frac{27}{\log n},
\]
where the last inequality holds for $k$ sufficiently large.

It is natural to ask if, for all posets $\mbf{P}$, we have
$RR(\mbf{P})\geq f(n)$, where $n$ is some reasonable parameter of the
poset---such as number of elements or width---and $f(n)$ tends to zero
fairly slowly with~$n$. In this section, we have shown that if such
an $f$ exists, it must be $O(1/\log n)$.

\section{Posets of fixed dimension}\label{sec:dim}

Because both order dimension and linear extension diameter focus on
linear extensions and reversing incomparable pairs, it is natural to
ask about the minimum value of $RR(\mbf{P})$ on a poset of dimension
$d$. An easy lower bound on the reversal ratio in terms of $d$ can be
obtained as follows. Let $\mc{R}$ be a realizer of size $d$ and
consider the sum
\[
\sum_{\substack{L_i,L_j\in \mc{R}\\i < j}}\dist(L_i,L_j).
\]
Each term is at most $\led(\mbf{P})$, and thus the sum is at most
$\binom{d}{2}\led(\mbf{P})$. On the other hand, each incomparable pair
must contribute 1 to at least $d-1$ terms in the sum, and so the sum is at
least $(d-1) \inc(\mbf{P})$.  Combining the two inequalities gives
$RR(\mbf{P})\geq 2/d$.

We define
$$
DRR(d) = \inf \{ RR(\mbf{P}) : \dim(\mbf{P}) \le d \}.
$$
The argument above shows that $DRR(d) \geq 2/d$ for each $d \ge 2$.
A poset $\mbf{P}$ with dimension $d\le 2$ is exactly one in which there is a pair of
linear extensions whose distance is equal to the number of incomparable
pairs, and so $RR(\mbf{P}) = 1$ for all posets $\mbf{P}$ of dimension at most~2,
and hence $DRR(2) = 1$.

In seeking upper bounds on $DRR(d)$ for $d \ge 3$, we need a $d$-dimensional
poset with small reversal ratio,
and we turn to the $d$-dimensional
grid $\mbf{n}^d$. This poset is the downset lattice of a disjoint
union of chains. In \cite{felsner:led-downlat}, Felsner and Massow
give an exact formula for the linear extension diameter of such
posets, from which we conclude that $\led(\mbf{n}^d)\sim n^{2d}/4$. On the
other hand, a straightforward calculation gives that
\[
\inc(\mbf{n}^d)\sim \frac{1}{2}n^{2d}\left(1-\frac{1}{2^{d-1}}\right).
\]
Hence, $RR(\mbf{n}^d)\to 1/(2(1-1/2^{d-1}))$ as $n \to \infty$.
This limit in turn tends to $1/2$ as $d\to\infty$, so these are not
good examples for large values of $d$.

For $d=3$, however, we deduce that $RR(\mbf{n}^3) \to 2/3$ as $n\to
\infty$, matching the lower bound.  Thus we have $DRR(3) = 2/3$. There
is already a gap between our upper and lower bounds for $d=4$, and we
don't know the value of $DRR(d)$ for any $d \ge 4$.  For $d=4$, the
bounds above give $1/2 \le DRR(4) \le 4/7$, and we know of no
improvements.

Perhaps of greater interest is the behavior of $DRR(d)$ as $d$
increases.  To move away from posets of small dimension, we recall the
posets $\mbf{P}_k = \mbf{P}_{k,1/\log 10}$ developed in Section~\ref{sec:prob}.  It is easy to see that
$\width(\mbf{P}_k) = k$ and so $\dim(\mbf{P}_k)\leq k$. Therefore $DRR(d) \le RR(\mbf{P}_d) \le 27/\log d$,
whenever $d$ is a sufficiently large multiple of~3.  (It is easy to deduce the same bound when $d$ is not
a multiple of~3.)  Our belief is that this upper bound is closer
to the true value of $DRR(d)$ than the lower bound of $2/d$. It is
worth noting that, for any $\varepsilon > 0$, one can show $\dim(\mbf{P}_{k, \varepsilon})\geq k^\delta$ for some
$0<\delta<1$, so having a more precise value for the dimension of
$\mbf{P}_k$ would not significantly improve the bounds on $DRR(d)$.

\section{Posets of fixed width}\label{sec:width}

Finally, we consider the minimum reversal ratio for posets of width
$w$.  In a similar manner to the previous section, we define
\[
WRR(w) = \inf \{ RR(\mbf{P}) : \width(\mbf{P}) \le w \}.
\]

Since $\dim(\mbf{P})\leq \width(\mbf{P})$ for every poset $\mbf{P}$,
we have $WRR(w) \ge DRR(w)$.
In particular, $WRR(2) = 1$, and the first interesting case is
that of width $3$; we know that $WRR(3) \geq DRR(3) = 2/3$.

A family of examples demonstrating that $WRR(3) \le 5/6$ is easy to come by.
Let $\mbf{W}_k$ denote a ``stack of $k$ standard examples'': we take antichains
$A_i = \{a_i,b_i,c_i\}$, for $i=0, \dots, k$, and, for each $i=1,\dots, k$, put
all of $A_{i-1}$ below all of $A_i$ except that the pairs $(a_{i-1},b_i)$,
$(b_{i-1},c_i)$ and $(c_{i-1},a_i)$ are incomparable.  For each $i$, no pair of
linear extensions can reverse all of the 6 incomparable pairs
$(a_i,b_i)$, $(a_i,c_i)$, $(b_i,c_i)$, $(a_{i-1},b_i)$, $(b_{i-1},c_i)$, $(c_{i-1},a_i)$.
Thus, of the $6k+3$ incomparable pairs in $\mbf{W}_k$, a pair of linear extensions
reverses at most $5k+3$.

However, we believe that neither the upper nor the lower bound is correct, and that
in fact $WRR(3)$ is equal to $3/4$.	We conclude by explaining the basis of this belief.

We construct a family of posets $\mbf{Z}_{t,k} = (Z,P)$, with $t\geq 3$ and
$k\geq 3$.  The set $Z$ consists of $t$ disjoint blocks $A_1, A_2, \dots, A_t$,
with each block $A_i$ containing $k$ elements arranged as a chain
$a_{i,1}<\cdots < a_{i,k}$.  For $j \ge i +2$, each element of $A_j$ is above
each element of $A_i$.	We also have comparabilities $a_{i,j} < a_{i+1,j}$, for
each $i =1, \dots t-1$ and $1\le j \le k$.  The remaining relations are those
induced by transitivity.

The poset $\mbf{Z}_{t,k}$ has width~3, with chains given by
$A_1<A_4<A_7< \cdots$, $A_2<A_5< \cdots$ and $A_3<A_6<\cdots$.	The only
incomparabilities are between pairs of blocks $A_i,A_\ell$, with $\ell=i+1$ or
$\ell=i+2$.  For such a pair $(i,\ell)$, we have $a_{i,j} < a_{\ell,m}$ if and
only if $j \le m$.  Therefore there are $(2t-3)\binom{k}{2}$ incomparable pairs.

The poset $\mbf{Z}_{9,4}$ is shown in Figure~\ref{fig:Z34}.

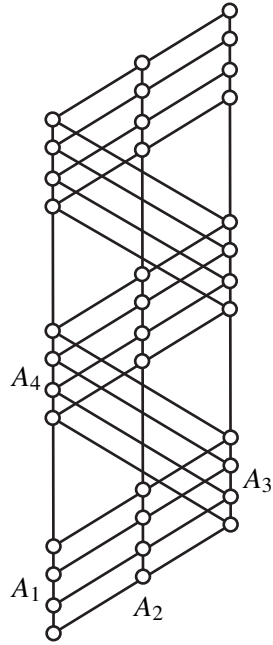
\begin{figure}
  \centering
  \begin{overpic}[]{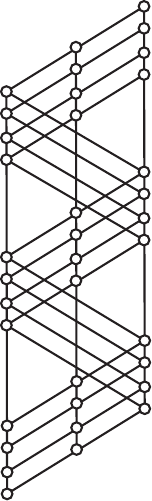}
    \put(-5,7){$A_1$}
    \put(14,4){$A_2$}
    \put(31,24){$A_3$}
    \put(-5,40){$A_4$}
  \end{overpic}
  \caption{The poset $\mbf{Z}_{9,4}$}
  \label{fig:Z34}
\end{figure}

\begin{thm}\label{thm:led-Ztk}
  The linear extension diameter of $\mbf{Z}_{t,k}$ is at least
  \[\left(t-1+\left\lceil\frac{t-2}{2}\right\rceil\right) \binom{k}{2}.\]
\end{thm}

\begin{proof}
  To show this lower bound on the linear extension diameter, we
  construct two linear extensions $L_1$ and $L_2$ of $\mbf{Z}_{t,k}$.

For a triple $A_\ell,A_{\ell+1},A_{\ell+2}$ of three consecutive blocks of $\mbf{Z}_{t,k}$,
we write $A_\ell/A_{\ell+1}/A_{\ell+2}$ for the following linear extension of the poset
restricted to $A_\ell\cup A_{\ell+1} \cup A_{\ell+2}$:
$$
a_{\ell,1} < a_{\ell+1,1} < a_{\ell+2,1} < a_{\ell,2} < a_{\ell+1,2} < a_{\ell+2,2} < a_{\ell,3} < \cdots
< a_{\ell,k} < a_{\ell+1,k} < a_{\ell+2,k}.
$$
For a consecutive pair of blocks, we define $A_\ell/A_{\ell+1}$ similarly, omitting the elements
$a_{\ell+2,j}$ from the order described above.

We will consider linear extensions $L$ of $\mbf{Z}_{t,k}$ in which blocks appear either consecutively or
in pairs and triples $A_\ell/A_{\ell+1}$ or $A_\ell/A_{\ell+1}/A_{\ell+2}$.
If, in some such linear extension $L$, $A_i$ and $A_j$ ($j=i+1$ or $j=i+2$) appear
together within an order $A_\ell/A_{\ell+1}/A_{\ell+2}$, we say that the pair $(A_i,A_j)$ is {\em entangled} in $L$.

Observe that, if the pair $(A_i,A_j)$ appears entangled in one linear extension, and in order in the other (i.e.,
with all of $A_i$ below all of $A_j$), then every pair of incomparable pairs between $A_i$ and $A_j$ is reversed
by the pair of linear extensions.

  We now form specific linear extensions of the form described that achieve the required lower bound.
  We start $L_1$ with $A_1/A_2/A_3$ and
  $L_2$ with $A_1$ in its entirety followed by $A_2$ and then $A_3/A_4/A_5$. The continuation
  of this pattern through the blocks appearing in $\mbf{Z}_{12,k}$ is
  shown below.
 \begin{center}
   \begin{tabular}{ccc}
     $L_1$ & \qquad & $L_2$\\\hline
     $A_{12}$  &\qquad& $A_{11}/A_{12}$\\
     $A_9/A_{10}/A_{11}$  &\qquad& $A_{10}$\\
     $A_8$ &\qquad&  $A_7/A_8/A_9$\\
     $A_5/A_6/A_7$  &\qquad& $A_6$\\
     $A_4$ &\qquad&  $A_3/A_4/A_5$\\
     $A_1/A_2/A_3$  &\qquad& $A_2$\\
     &\qquad&  $A_1$
   \end{tabular}
 \end{center}
 More formally, the linear extension $L_1$ contains each linear order $A_{4i-3}/A_{4i-2}/A_{4i-1}$, for $i \ge 1$, and the blocks $A_{4i}$
 ($i\ge 1$) appear separately, above all blocks with lower index and below all blocks with higher index.  If $t$ is congruent to 1 or 2 modulo~4,
 so that some designated triple $A_{4i-3}/A_{4i-2}/A_{4i-1}$ includes $A_t$ and supposedly some non-existent blocks with higher index, then the
 blocks beyond $A_t$ are omitted.  Similarly, $L_2$ includes each linear order $A_{4i-1}/A_{4i}/A_{4i+1}$ ($i\ge 1$), with the blocks
 $A_{4i+2}$ appearing singly: again, blocks with index outside $\{1, \dots, t\}$ are omitted.

 To determine the number of incomparable pairs reversed between $L_1$ and $L_2$, we
 first observe that, for blocks $A_i$ and $A_j$, with $j=i+1$ or $j=i+2$, the pair $(A_i,A_j)$
 is entangled in at most one of the linear extensions.
Moreover, each of the $t-1$ consecutive pairs $(A_i,A_{i+1})$ is entangled in exactly one of the two
linear extensions, so the incomparable pairs between such a pair of blocks are all reversed.
For the pairs at distance~2, the $\lceil (t-2)/2 \rceil$ pairs $(A_i,A_{i+2})$ with $i$ odd are
entangled in one of the two linear extensions, but the pairs with $i$ even are not.
Therefore,
 \[\dist(L_1,L_2) =
 \left(t-1+\left\lceil\frac{t-2}{2}\right\rceil\right)\binom{k}{2},\]
 and consequently this is a lower bound on $\led(\mbf{Z}_{t,k})$, as claimed.
\end{proof}

Since $\inc(\mbf{Z}_{t,k}) = (2t-3)\binom{k}{2}$,
Theorem~\ref{thm:led-Ztk} implies that $RR(\mbf{Z}_{t,k})\geq 3/4$. We
believe the inequality of Theorem~\ref{thm:led-Ztk} to be an equality
for $k$ and $t$ sufficiently large, which would imply that $WRR(3) \ge 3/4$.
Moreover, we suspect that this family of examples is (essentially) extremal.

\begin{conj*}
  $WRR(3) = 3/4$.
\end{conj*}

We can form similar families of posets for higher values of the width, and it is
possible that these remain extremal.

\section{Acknowledgements} 

We thank the referees for a careful reading, leading to substantial improvements
in the paper.  



\begin{thebibliography}{1}

\bibitem{albert-frieze}
{\sc Albert, M.~H., and Frieze, A.~M.}
\newblock Random graph orders.  
\newblock {\em Order 6}, 1 (1989), 19--30.

\bibitem{bollobas-brightwell:width-rgos}
{\sc Bollob\'as, B., and Brightwell, G.}
\newblock The width of random graph orders, 
\newblock {\em The Mathematical Scientist, 20} (1995), 69--90. 

\bibitem{bollobas-brightwell:dimension-rgos}
{\sc Bollob\'as, B., and Brightwell, G.} 
\newblock The dimension of random graph orders, 
\newblock in {\em The Mathematics of Paul Erd\H os II}, 
R. L. Graham and J. Ne\v set\v ril, eds., Springer-Verlag, 1996, pp.51--69.

\bibitem{brightwell:led-diampairs}
{\sc Brightwell, G., and Massow, M.}
\newblock Diametral pairs of linear extensions.
\newblock Submitted., 2010.

\bibitem{felsner:led-downlat}
{\sc Felsner, S., and Massow, M.}
\newblock Linear extension diameter of downset lattices of two-dimensional
  posets.
\newblock {\em SIAM J. Discrete Math. 25}, 1 (2011), 112--129.

\bibitem{felsner:led}
{\sc Felsner, S., and Reuter, K.}
\newblock The linear extension diameter of a poset.
\newblock {\em SIAM J. Discrete Math. 12}, 3 (1999), 360--373.

\bibitem{fox}
{\sc Fox, J.} 
\newblock A bipartite analogue of Dilworth's Theorem. 
\newblock {\em Order 23}, 2-3 (2006), 197--209.

\bibitem{lubotzky-phillips-sarnak}
{\sc Lubotzky, A., Phillips, R., and Sarnak, P.} 
\newblock Ramanujan graphs.  
\newblock {\em Combinatorica  8}, 3 (1988), 261--277.

\bibitem{trotter:dimbook}
{\sc Trotter, W.~T.}
\newblock {\em Combinatorics and partially ordered sets: Dimension theory}.
\newblock Johns Hopkins Series in the Mathematical Sciences. Johns Hopkins
  University Press, Baltimore, MD, 1992.

\end{thebibliography}
\end{document}